\documentclass[a4paper,14pt,reqno,intlimits]{extarticle}
\usepackage[cp1251]{inputenc}
\usepackage[T2A]{fontenc}
\usepackage[russian]{babel}
\usepackage{amsfonts,amsmath,amsxtra,amsthm,amssymb,latexsym,mathrsfs,graphicx,mathtools}
\usepackage{multirow}

\usepackage[ruled]{algorithm2e}

\SetAlCapHSkip{0pt}
\IncMargin{-\parindent}

\theoremstyle{plain}
\newtheorem{theorem}{Теорема}

\newtheorem{definition}{Определение}

\theoremstyle{definition}
\newtheorem{remark}{Замечание}
\newtheorem{example}{Пример}

\numberwithin{theorem}{section}
\numberwithin{lemma}{section}
\numberwithin{proposition}{section}
\numberwithin{corollary}{section}
\numberwithin{remark}{section}
\numberwithin{definition}{section}
\numberwithin{example}{section}
\numberwithin{tasks}{section}

\DeclarePairedDelimiter\bracket{(}{)}
\newcommand{\br}[1]{\bracket*{#1}}
\DeclarePairedDelimiter\figbracket{\{}{\}}
\newcommand{\fbr}[1]{\figbracket*{#1}}

\DeclareMathOperator{\argmin}{arg\,min}

\author{Ф. С. Стонякин}
\title{Об адаптивном проксимальном методе для некоторого класса вариационных неравенств и смежных задач}

\begin{document}
\maketitle

\begin{abstract}
Для задач безусловной оптимизации хорошо известна концепция неточного оракула, предложенная О. Деволдером, Ф. Глинером и Ю.Е. Нестеровым. В настоящей работе введен аналог понятия неточного оракула (модели функции) для абстрактных задач равновесия, вариационных неравенств и седловых задач. Это позволило предложить аналог известного проксимального метода А.С. Немировского для вариационных неравенств с адаптивной настройкой на уровень гладкости для достаточно широкого класса задач. При этом предусмотрена возможность неточного решения вспомогательных задач проектирования на итерациях метода. Показано, что возникающие погрешности не накапливаются в ходе работы метода. Получены оценки скорости сходимости предложенного метода. Обоснована оптимальность метода с точки зрения теории нижних оракульных оценок. Показано, что предложенный метод применим к смешанным вариационным неравенствам и композитным седловым задачам. Приведен пример, демонстрирующий возможность существенного повышения качества работы метода по сравнению с теоретическими оценками за счет адаптивности критерия остановки.
\end{abstract}

\section{Введение}

Вариационные неравенства (ВН) и седловые задачи часто возникают в самых разных проблемах оптимизации и имеют многочисленные приложения в математической экономике, математическом моделировании транспортных потоков, теории игр и других разделах математики (см., например, \cite{FaccPang_2003}). Исследования в области алгоритмических методов решения вариационных неравенств и седловых задач активно продолжаются (см., например, \cite{Antipin_2000}--\cite{Nesterov_Doct}). Наиболее известным аналогом градиентного метода для ВН является экстраградиентный метод Г.М. Корпелевич \cite{Korpelevich}, в качестве одного из современных вариантов которого можно выделить проксимальный зеркальный метод А.С. Немировского \cite{Nemirovski_2004}. В работе предложен аналог этого метода на базе ряда возникших в последние годы новых идей в области алгоритмической оптимизации, о которых мы скажем несколько слов \cite{Nesterov_2015, Gasn_2017, Ston_2017, Pap_Model}.

Некоторое время назад Ю.Е. Нестеровым в \cite{Nesterov_2015} предложены так называемые {\it универсальные} градиентные методы для задач выпуклой минимизации. Под универсальностью метода понимается возможность адаптивной настройкой метода на уровень гладкости задачи, что может позволить ускорять работу метода (\cite{Nesterov_2015}, см. также раздел 5 из \cite{Gasn_2017}).

Недавно нами совместно с А.В. Гасниковым, П.Е. Двуреченским и А.А. Титовым был предложен универсальный проксимальный аналог метода А.С. Немировского для решения вариационных неравенств (см. замечание 5.1 из \cite{Gasn_2017} и статью \cite{Ston_2017}). При этом предложенная методика позволяет учесть возможность неточного задания оракула для оператора поля. Для оператора $g: Q \rightarrow \mathbb{R}^n$, заданного на выпуклом компакте $Q\subset \mathbb{R}^n$ под {\it сильным вариационным неравенством} понимаем неравенство вида
\begin{equation}\label{Old_1}
  \langle g(x_*),\ x_*-x\rangle \leqslant 0,
\end{equation}
где $g$ удовлетворяет условию Гельдера:
\begin{equation}\label{eq101}
||g(x)-g(y)||_{\ast}\leqslant L_{\nu}||x-y||^{\nu} \quad \forall x,y\in Q
\end{equation}
для произвольного $\nu\in[0;1]$, причем $L_{0}<+\infty$ (другие константы $L_{\nu}$ ($\nu \neq 0$) могут быть бесконечными).

Отметим, что в \eqref{Old_1} требуется найти $x_* \in Q$ (это $x_*$ и называется решением ВН) для которого
\begin{equation}\label{Old_2}
  \max_{x\in Q}\langle g(x_*),\ x_*-x\rangle \leqslant 0.
\end{equation}
Для монотонного оператора поля $g$ можно рассматривать {\it слабые вариационные неравенства}
\begin{equation}\label{Old_11}
  \langle g(x),\ x_*-x\rangle \leqslant 0.
\end{equation}
Обычно в \eqref{Old_11} требуется найти $x_* \in Q$, для которого \eqref{Old_11} верно при всех $x \in Q$.
Предложенная в \cite{Gasn_2017, Ston_2017} методика позволяет получить приближённое решение задач \eqref{Old_1} -- \eqref{Old_11} с точностью $\varepsilon$ с оценкой сложности (достаточного количества итераций для достижения приемлемого качества решения)
\begin{equation}\label{eq_opt_estim}
O\left(\left(\frac{1}{\varepsilon}\right)^{\frac{2}{1+\nu}}\right),
\end{equation}
которая оптимальна при $\nu = 0$ и $\nu = 1$ \cite{GuzNem_2015, NemYud_1979, Nemirovski_1994_95, Optimal}.
При этом адаптивность метода \cite{Gasn_2017, Ston_2017} на практике может приводить к ускорению работы метода по сравнению с оценками \eqref{eq_opt_estim} (подробнее об этом написано в замечании \ref{Rem_Steiner} ниже).

Наконец, недавно А. В. Гасниковым в разделе 3 пособия \cite{Gasn_2017} (см. также \cite{Pap_Model}) предложена абстрактная концепция {\it ($\delta, L$)-модели} функции, которая является прямым обобщением известного понятия \textit{$(\delta, L)$-оракула} О. Деволдера -- Ф. Глинера -- Ю.Е. Нестерова для задач безусловной оптимизации (см. работу \cite{a4}, которая активно цитируется оптимизационным сообществом).

\begin{definition}
\label{delta_L_oracle}
Пара $(f_{\delta}(y), g_{\delta}(y)) \in \mathbb{R} \times E^*$ есть $(\delta, L)$-оракул для функции $f: Q \rightarrow \mathbb{R}$ в точке $y$, если для любого $x \in Q$ справедливо неравенство
\begin{gather}
\label{DLoracleineq}
0 \leq f(x) - f_{\delta}(y) - \langle g_{\delta}(y), x - y \rangle \leq \frac{L}{2} \|x - y\|^2 + \delta,
\end{gather}
\end{definition}
Отметим, что здесь $E$ --- пространство, содержащее аргументы функции $f$, а $E^*$ --- пространство, сопряжённое (двойственное) к $E$. Концепция $(\delta, L)$-модели функции \cite{a4} в точке отличается от предложенного О. Деволдером, Ю. Е. Нестеровым и Ф. Глинером понятия неточного оракула тем, что линейная функция $\langle g_{\delta}(y), x - y \rangle$ в \eqref{DLoracleineq} заменяется на некоторую абстрактную выпуклую функцию.
\begin{definition}
\label{gen_delta_L_oracle}
Будем говорить, что имеется $(\delta, L)$-модель функции $f(x)$ в точке $y$, и обозначать эту модель $(f_{\delta}(y), \psi_{\delta}(x,y))$, если для любого $x \in Q$ справедливо неравенство
\begin{gather}
\label{exitLDLOrig}
0 \leq f(x) - f_{\delta}(y) - \psi_{\delta}(x,y) \leq \frac{L}{2} \|x - y\|^2 + \delta,
\end{gather}
\begin{gather}
\label{exitLDLOrig_eqaul}
\psi_{\delta}(x,x)=0 \,\,\, \forall x \in Q
\end{gather}
и $\psi_{\delta}(x,y)$~--- выпуклая функция по $x$ для $\forall y \in Q$.
\end{definition}

Отметим, что концепция ($\delta, L$)-модели функции из определения \ref{gen_delta_L_oracle} позволяет обосновать сходимость обычного и быстрого градиентного методов выпуклой минимизации для достаточно широкого класса задач \cite{Gasn_2017,Pap_Model, Ston_Model}.

Настоящая заметка посвящена новому развитию упомянутых выше идей. Мы рассматриваем абстрактную постановку задачи о нахождении точек равновесия и обосновываем возможность использования аналога адаптивного (универсального) метода для вариационных неравенств из \cite{Gasn_2017, Ston_2017} для такой постановки. На базе этой идеологии вводится аналог концепции \textit{$(\delta, L)$-модели} функции для седловых задач. Такой подход позволит распространить методику \cite{Gasn_2017, Ston_2017} на более широкий класс задач, в частности {\it смешанные вариационные неравенства} \cite{Konnov_2017, Bao_Khanh} и {\it композитные седловые задачи} (здесь стоит отметить популярную в оптимизационном сообществе работу \cite{Chambolle}). Далее будем рассматривать задачу нахождения решения $x_*\in Q$ абстрактной задачи равновесия (неравенство Фань Цзы)
\begin{equation}\label{eq13}
\psi(x,x_*)\geqslant 0 \quad \forall x\in Q
\end{equation}
для некоторого выпуклого компакта $Q\subset\mathbb{R}^n$, а также функционала $\psi:Q\times Q\rightarrow\mathbb{R}$. Если предположить абстрактную монотонность функционала $\psi$:
\begin{equation}\label{eq14}
\psi(x,y)+\psi(y,x)\leqslant0\;\;\forall x,y\in Q,
\end{equation}
то всякое решение \eqref{eq13} будет также и решением двойственной задачи равновесия
\begin{equation}\label{eq115}
\psi(x_*,x)\leqslant 0 \quad \forall x\in Q.
\end{equation}
В общем случае сделаем предположение о существовании решения $x_*$ задачи \eqref{eq13}. Приведем пару примеров задания $\psi$, для которых данное условие заведомо выполнено.

\begin{example}
Если для некоторого оператора $g:R\rightarrow\mathbb{R}^n$ положить
\begin{equation}\label{eq16}
\psi(x,y)=\langle g(y),x-y\rangle\;\;\forall x,y\in Q,
\end{equation}
то \eqref{eq13} и \eqref{eq115} будут равносильны соответственно стандартным сильному и слабому вариационному неравенству с оператором $g$.
\end{example}

\begin{example}\label{eq: Mixed_VI}
Для некоторого оператора $g:Q\rightarrow\mathbb{R}^n$ и выпуклого функционала $h:Q\rightarrow\mathbb{R}^n$ простой структуры (см. например, \cite{Gasn_2017}) выбор функционала
\begin{equation}\label{eq17}
\psi(x,y)=\langle g(y),x-y\rangle+h(x)-h(y)
\end{equation}
приводит к {\it смешанному вариационному неравенству} \cite{Konnov_2017, Bao_Khanh}
\begin{equation}\label{eq18}
\langle g(y),y-x\rangle+h(y)-h(x)\leqslant0,
\end{equation}
которое в случае монотонности оператора $g$ влечет
\begin{equation}\label{eq19}
\langle g(x),y-x\rangle+h(y)-h(x)\leqslant0.
\end{equation}
\end{example}

Отметим, что известно немало проксимальных методов для задач нахождения точек равновесия (см. в частности \cite{Antipin_equlib, Semenov} и имеющуюся там библиографию). В частности в \cite{Semenov} предложен проксимальный метод для задач равновесия в гильбертовых пространствах. Однако, как правило, в этих работах лишь исследуются условия сходимости предлагаемых методов без какого-либо обоснования оптимальности скорости сходимости, а также критериев остановки рассматриваемых методов, гарантирующих достижение приемлемого качества решения. Мы же в данной заметке предлагаем для таких задач предлагаем аналог метода А.С. Немировского \cite{Nemirovski_2004}. При этом предлагаемый нами адаптивный критерий остановки за конечное число шагов обеспечивает достижение приемлемого качества приближенного решения по аналогии со случаем вариационных неравенств, рассмотренным в \cite{Nemirovski_2004}.

\begin{remark}
Хорошо известно, что множества решений \eqref{eq13} и \eqref{eq115} совпадают в предположениях $\psi(x,x) = 0$ для всякой точки $x \in Q$, абстрактной монотонности \eqref{eq14}, выпуклости $\psi$ по первой переменной, а также полунепрерывности $\psi$ снизу по первой переменной и слабой полунепрерывности сверху по второй переменной.
\end{remark}

Выделим основные результаты данной работы:

- Как обобщение известного для задач оптимизации понятия $(\delta, L)$-оракула О. Деволдера -- Ф.Глинера -- Ю.Е. Нестерова предложено понятие $(\delta, L)$-модели для абстрактной задачи нахождения точек равновесия.

- Для задач равновесия предложен аналог известного для вариационных неравенств и седловых задач проксимального метода А.С. Немировского (алгоритм \ref{alg:SIGM}). При этом рассмотрено специальное условие гладкости \eqref{eq20}, а также предложен адаптивный критерий остановки метода. Адаптивность позволяет применять метод для задач с неизвестной константой $L$, а также может приводить к ускорению достижения желаемой точности решения.

- Получена оценка скорости сходимости алгоритма \ref{alg:SIGM}, указывающая на его оптимальность с точки зрения теории нижних оракульных оценок (теорема \ref{th1} и замечание \ref{rem_optimal}). При этом учитываются погрешности задания функционала $\psi$, а также решения вспомогательных задач на итерациях метода (см. \eqref{eq22}). Доказано, что погрешности обоих типов не накапливаются в ходе работы метода.

- Введено понятие $(\delta, L)$-модели для седловых задач (определение \ref{Def_Sedlo}). Получена оценка скорости сходимости алгоритма \ref{alg:SIGM} для седловых задач, допускающих $(\delta, L)$-модель (теорема \ref{th_sedlo}). При этом также показана возможность учёта погрешности задания модели седловой задачи, а также погрешности решения вспомогательных задач на итерациях метода.

- Обоснована возможность применения предложенного метода к смешанным вариационным неравенствам (пример \ref{eq: Mixed_VI}) и композитным седловым задачам
(пример \ref{Example_composite}).

\section{Адаптивный проксимальный метод для абстрактных задач равновесия}
Мы предлагаем адаптивный проксимальный метод для задач \eqref{eq13} и \eqref{eq115} при следующих {\bf предположениях} для функционала $\psi$:
\begin{enumerate}
\item[(i)] функционал $\psi(x,y)$ выпуклый по первой переменной;
\item[(ii)] $\psi(x,x)=0\;\;\forall x\in Q$;
\item[(iii)] ({\it абстрактная монотонность}) неравенство \eqref{eq14}
\item[(iv)] ({\it обобщенная гладкость})
\begin{equation}\label{eq20}
\psi(x,y)\leqslant\psi(x,z)+\psi(z,y)+LV(x,z)+LV(z,y)+\delta
\end{equation}
для некоторой фиксированной константы $L>0$, где $\delta>0$~--- некоторая постоянная величина (оценка погрешности задания $\psi$,
степень отклонения от гладкости), а $V(x, y)$ есть расхождение Брэгмана, порожденное $1$-сильно выпуклой прокс-функцией $d:Q\rightarrow\mathbb{R}$ (относительно нормы $||\cdot||$).
\end{enumerate}

Напомним, что \emph{расхождение Брэгмана} широко используется в оптимизации и вводится на базе конечной 1-сильно выпуклой функции $d$ (порождает расстояния), которая дифференцируема во всех точках $x\in Q$:
\begin{equation}\label{3}
V(x,y)=d(x)-d(y)-\langle \nabla d(y),x-y\rangle \quad \forall x,y\in Q,
\end{equation}
где $\langle \cdot,\cdot\rangle$ --- скалярное произведение в $\mathbb{R}^n$. В случае стандартной евклидовой нормы $\|\cdot\|_2$ и расстояния в $\mathbb{R}^n$ можно считать, что:
$$
V(x, y) = d(x - y) = \frac{1}{2} \|x - y\|_2^2 \quad \forall x, y \in Q.
$$
Однако часто возникает необходимость использовать и неевклидовы нормы. В этом случае примеры и свойства дивергенции Брэгмана уже не столь не тривиальны (см., например, раздел 2 пособия \cite{Gasn_2017}).

Отметим, что в случае обычного ВН \eqref{eq16} и евклидовой нормы условие \eqref{eq20} сводится к неравенству
\begin{equation}\label{eq21}
\langle g(z)-g(y),z-x\rangle\leqslant\frac{L}{2}||z-x||^2+\frac{L}{2}||z-y||^2+\delta.
\end{equation}
При $\delta=0$ неравенство \eqref{eq21} легко проверяется, например, для оператора $g(x)=\nabla f(x)$, где $f:Q\rightarrow\mathbb{R}$ есть некоторый выпуклый субдифференцируемый функционал и $\nabla f(x)$ --- произвольный субградиент $f$ в точке $x$. Заметим, что при $\delta=0$ похожее на \eqref{eq20} условие
\begin{equation}\label{eq200}
\psi(x,y)\leqslant\psi(x,z)+\psi(z,y)+ a\|y - z\|^2 + b\|z - x\|^2 \quad  \forall x, y, z \in Q
\end{equation}
($a$ и $b$ --- положительные константы) предложено в \cite{Mastroeni} и использовалось во многих  последующих работах (см., например \cite{Semenov} и имеющуюся там библиографию). Наш подход позволяет работать с неевклидовой прокс-структурой, а также учитывать неточность $\delta$, что важно для идеологии универсальных методов \cite{Nesterov_2015, Gasn_2017, Ston_2017}.

На \eqref{eq21} основан предложенный ранее проксимальный метод для вариационных неравенств (см. замечание 5.1 из \cite{Gasn_2017}, а также статью \cite{Ston_2017}). Естественно возникает идея обобщить этот метод на абстрактные задачи \eqref{eq13} и \eqref{eq115} в предположениях их разрешимости, а также (i)--(iv). При этом будем учитывать погрешность $\delta$ в \eqref{eq20}, а также погрешность $\tilde{\delta}$ решения вспомогательных задач на итерациях согласно одному из достаточно известных в алгоритмической оптимизации подходов (см. например, раздел 3 из \cite{Gasn_2017}, а также \cite{Pap_Model, Ston_Model}):
\begin{equation}\label{eq22}
x:=\argmin_{y\in Q}^{\tilde{\delta}}\varphi(y),\text{ если }\langle\nabla\varphi(x),x-y\rangle\leqslant\tilde{\delta}.
\end{equation}

Опишем $(N+1)$-ую итерацию рассматриваемого метода ($N=0,1,2,\ldots$), выбрав начальное приближение
$$x^0=\argmin\limits_{x\in Q}d(x),$$
зафиксировав точность $\varepsilon>0$, а также некоторую константу $L^0\leqslant2L$.
\begin{algorithm}[h!]\label{alg:SIGM}
\caption{Адаптивный метод для абстрактных задач равновесия.}
\begin{itemize}
\item[1.] $N:=N+1$; $L^{N+1}:=\frac{L^N}{2}$.
\item[2.] Вычисляем:\\
$y^{N+1}:=\argmin\limits_{x\in Q}^{\tilde{\delta}}\fbr{\psi(x, x^N)+L^{N+1}V(x,x^N)}$,\\
$x^{N+1}:=\argmin\limits_{x\in Q}^{\tilde{\delta}}\fbr{\psi(x, y^{N+1})+L^{N+1}V(x,x^N)}$\\
до тех пор, пока не будет выполнено:
\begin{equation}\label{eq23}
\begin{split}
\psi(x^{N+1}, x^N)\leqslant\psi(y^{N+1}, x^N)+\psi(x^{N+1}, y^{N+1})+\\
+L^{N+1}V(y^{N+1}, x^N)+L^{N+1}V(y^{N+1}, y^{N+1})+\delta.
\end{split}
\end{equation}
\item[3.] \textbf{Если} \eqref{eq23} не выполнено, \textbf{то} $L^{N+1}:=2L^{N+1}$ и повторяем п. 2.
\item[4.] \textbf{Иначе} переход к п.~1.
\item[5.] Критерий остановки метода:
\begin{equation}\label{eq24}
\sum_{k=0}^{N-1}\frac{1}{L^{k+1}}\geqslant\frac{\max_{x \in Q}V(x,x^0)}{\varepsilon}.
\end{equation}
\end{itemize}
\end{algorithm}

Для краткости будем всюду далее обозначать
\begin{equation}\label{eq25}
S_N:=\sum\limits_{k=0}^{N-1}\frac{1}{L^{k+1}}.
\end{equation}

Справедлива следующая
\begin{theorem}\label{th1}
После остановки рассматриваемого метода для всякого $x \in Q$ будет заведомо выполнено неравенство:
\begin{equation}\label{eq26}
-\frac{1}{S_N}\sum_{k=0}^{N-1}\frac{\psi(x,y^{k+1})}{L^{k+1}} \leqslant\frac{V(x, x^0)}{S_N}+2\tilde{\delta}+\delta \leqslant \varepsilon+2\tilde{\delta}+\delta,
\end{equation}
a также
\begin{equation}\label{eq27}
\psi(\tilde{y},x)\leqslant\frac{V(x,x^0)}{S_N}+2\tilde{\delta}+\delta \leqslant \varepsilon+2\tilde{\delta}+\delta
\end{equation}
при
\begin{equation}\label{eq28}
\tilde{y}:=\frac{1}{S_N}\sum_{k=0}^{N-1}\frac{y^{k+1}}{L^{k+1}}.
\end{equation}
\end{theorem}

\begin{proof}
После завершения $(N+1)$-ой итерации метода ($N=0,1,2\ldots$) ввиду \eqref{eq22} имеем:
$$\psi(y^{N+1}, x^N)\leqslant\psi(x^{N+1},x^N)+L^{N+1}V(x^{N+1},x^N)-$$
$$-L^{N+1}V(x^{N+1},y^{N+1})-L^{N+1}V(y^{N+1},x^N)+\tilde{\delta},$$
$$\psi(x^{N+1},y^{N+1})\leqslant\psi(x,y^{N+1})+L^{N+1}V(x,x^N)-$$
$$-L^{N+1}V(x,x^{N+1})-L^{N+1}V(x^{N+1},x^N)+\tilde{\delta}.$$

Далее, в силу \eqref{eq23}:
$$
\psi(x^{N+1}, x^N)\leqslant\psi(y^{N+1}, x^N)+\psi(x^{N+1}, y^{N+1})+
$$
$$
+L^{N+1}V(y^{N+1}, x^N)+L^{N+1}V(y^{N+1}, y^{N+1})+\delta.
$$
Отметим, что предположение \eqref{eq20} гарантирует выполнение условия \eqref{eq23} при $L^{N+1} \geqslant L$ после нескольких увеличений
$L^{N+1}$ в 2 раза. Просуммировав последние три неравенства, получаем:
$$\psi(x,y^{N+1})+L^{N+1}V(x,x^N)-L^{N+1}V(x,x^{N+1})\leqslant2\tilde{\delta}+\delta,$$
откуда
$$\frac{\psi(x,y^{N+1})}{L^{N+1}}+V(x,x^N)-V(x,x^{N+1})\leqslant\frac{1}{L^{N+1}}(2\tilde{\delta}+\delta),$$
или после суммирования:
$$-\sum_{k=0}^{N-1}\frac{\psi(x,y^{k+1})}{L^{k+1}}\leqslant\sum_{k=0}^{N-1}\br{V(x,x^k)-V(x,x^{k+1})}+(2\tilde{\delta}+\delta)S_N,$$
откуда и следует доказываемое неравенство \eqref{eq26}.
\end{proof}

\begin{remark}\label{rem_optimal}
Ввиду \eqref{eq20} и выбора $L^0\leqslant2L$ гарантированно будет верно $$L^{k+1}\leqslant2L\;\;\forall k=\overline{0,N-1}.$$
Поэтому $$ S_N\geqslant\frac{N}{2L}$$ и \eqref{eq26}--\eqref{eq27} означают, что для всякого $x \in Q$ будут верны неравенства:
\begin{equation}\label{eq30}
\psi(\tilde{y},x)\leqslant-\frac{1}{S_N}\sum_{k=0}^{N-1}\frac{\psi(x,y^{k+1})}{L^{k+1}}\leqslant\frac{2LV(x, x_0)}{N}+2\tilde{\delta}+\delta\leqslant\varepsilon+2\tilde{\delta}+\delta
\end{equation}
после выполнения не более, чем
\begin{equation}\label{eqv_5.5}
O\left(\frac{1}{\varepsilon}\right)
\end{equation}
итераций предлагаемого метода. При этом нетрудно проверить, количество решений вспомогательных задач в п.2 алгоритма \ref{alg:SIGM} на $N$ итерациях метода не превышает
$$
2N + \log_{2} \frac{L}{L^0},
$$
т.е. стоимость итерации в среднем будет сопоставимой со стоимостью итерации классического экстраградиентного метода, предполагающей решение двух вспомогательных задач на каждой итерации. Отметим, что оценка \eqref{eqv_5.5} с точностью до числового множителя оптимальна для вариационных неравенств и седловых задач \cite{GuzNem_2015, NemYud_1979, Nemirovski_1994_95, Optimal}. Предложенный алгоритм \ref{alg:SIGM} применим и для более широкого класса задач равновесия для функционала $\psi$, удовлетворяющего предположениям (i) -- (iv) выше и с точки зрения количества итераций будет оптимальным с точностью до числового множителя.
\end{remark}

\begin{remark}
Для обычных слабых вариационных неравенств \eqref{Old_11} неравенство \eqref{eq30} можно заменить на
\begin{equation}\label{eq14krit}
\max_{x\in Q} \left \langle g(x), \tilde{y} -x \right\rangle \leqslant \varepsilon + 2 \widetilde{\delta} + \delta.
\end{equation}
Отметим, что именно \eqref{eq14krit} обычно используют как критерий качества решения вариационного неравенства (см., например \cite{Nemirovski_2004}).
\end{remark}

\begin{remark}\label{Rem_Steiner}
В случае гёльдерова оператора поля $g$ (удовлетворяющего \eqref{eq101}) известно неравенство (см. замечание 5.1 из \cite{Gasn_2017}, а также работу \cite{Ston_2017})
\begin{equation}\label{eq21hold}
\langle g(z)-g(y),z-x\rangle\leqslant\frac{L}{2}||z-x||^2+\frac{L}{2}||z-y||^2+ \frac{\varepsilon}{2}
\end{equation}
для некоторой константы $L$, зависящей от $\varepsilon$. На базе интерполяции \eqref{eq21hold} алгоритм 1 сводится к
{\it универсальному} методу для вариационных неравенств \cite{Ston_2017}. Этот метод предполагает адаптивную настройку на уровень гладкости оператора $g$. Заметим, что полученная нами теорема \ref{th1} позволяет обобщить этот подход на смешанные вариационные неравенства (пример \ref{eq: Mixed_VI}) с гёльдеровым оператором $g$.

Оказывается, что настройка метода на уровень гладкости оператора может позволить ускорить работу метода. В частности, для негладкой задачи ($\nu = 0$) метод на практике может работать со скоростью $O\left(\frac{1}{\varepsilon}\right)$, оптимальной для задач с липшицевым оператором ($\nu = 1$). Приведём один такой пример, связанный с известной задачей {\it Ферма-Торричелли-Штейнера}, но при наличии нескольких негладких функциональных ограничений. В общем случае задачи с функциональными ограничениями с использованием принципа множителей Лагранжа сводятся к седловым задачам и соответствующих им вариационным неравенствам.

\textbf{Задача.} Для заданного набора точек $A_k=(a_{1k},a_{2k},\ldots,a_{nk},)$ n-мерного евклидова пространства $\mathbb{R}^n$ найти точку $X=(x_1,x_2,\ldots,x_n)$ такую, что значение целевой функции
$$f(x):=\sum\limits_{k=1}^n\sqrt{(x_1-a_{1k})^2+(x_2-a_{2k})^2+\ldots+(x_n-a_{nk})^2}$$
будет минимальным для всех точек, удовлетворяющих системе негладких функциональных ограничений:
$$\varphi_1(x)=\alpha_{11}|x_1|+\alpha_{12}|x_2|+\ldots+\alpha_{1n}|x_n|-1 \leq 0,$$
$$\varphi_2(x)=\alpha_{21}|x_1|+\alpha_{22}|x_2|+\ldots+\alpha_{2n}|x_n|-1 \leq 0,$$
$$\ldots$$
$$\varphi_m(x)=\alpha_{m1}|x_1|+\alpha_{m2}|x_2|+\ldots+\alpha_{mn}|x_n|-1 \leq 0.$$

При этом коэффициенты $\alpha_{11},\alpha_{12},\ldots,\alpha_{mn}$ образуют матрицу
$$
\begin{pmatrix}
    \alpha_{11} & \alpha_{13} & \dots & \alpha_{1n} \\
    \alpha_{21} & \alpha_{23} & \dots & \alpha_{2n} \\
    \hdotsfor{4} \\
    \alpha_{m1} & \alpha_{m3} & \dots & \alpha_{mn}
\end{pmatrix}
$$
Один из элементов каждой строки этой матрицы содержится в промежутке $(1,10)$, а оставшиеся элементы равны 1. Положим $$L(x,\lambda)=f(x)+\sum\limits_{p=1}^m\lambda_p\varphi_p(x),\;\overrightarrow{\lambda}=(\lambda_1,\lambda_2,\ldots,\lambda_m)$$
и рассмотрим вариационное неравенство:
$$\langle g(x_*,\overrightarrow{\lambda}_*),(x_*,\overrightarrow{\lambda}_*)-(x,\overrightarrow{\lambda})\rangle\leqslant0\;\;\forall(x,\overrightarrow{\lambda})\in B\subset\mathbb{R}^{n+m},$$
где
$$
B=\left\{(x,\overrightarrow{\lambda})\,|\,\sum\limits_{k=1}^nx_k^2+\sum\limits_{p=1}^m\lambda_p^2\leqslant1\right\},
$$
$$
g(x,\lambda)=
\begin{pmatrix}
    \nabla f(x)+\sum\limits_{p=1}^m\lambda_p\nabla\varphi_p(x), \\
    \varphi_1(x),\varphi_2(x),\ldots,\varphi_m(x)
\end{pmatrix}.
$$

Нетрудно понять, что целевой функционал и ограничения недифференцируемы в бесконечном множестве точек, но имеют ограниченные субдифференциалы в смысле выпуклого анализа. Поэтому оператор $g$ удовлетворяет \eqref{eq101} при $\nu = 0$.

Рассмотрим случай $n=10$ переменных и $m=100$ ограничений, начальное приближение
$$x^0= \frac{(0.2,0.2,\ldots,0.2)}{\|(0.2,0.2,\ldots,0.2)\|} \in \mathbb{R}^{n+m},$$
норму и расстояние примем стандартным евклидовым. При этом координаты точек $A_k= (a_{1k}, a_{2k}, \ldots, a_{nk})$ для $k=1,2,\ldots,5$ выберем как столбцы следующей мaтрицы
$$
A=
\begin{pmatrix}
5&4&-7&-2&-3&-8&5&3&8&4\\
-7&-8&-9&-8&-8&6&-4&-8&4&-3\\
-4&-5&8&9&-5&-4&-9&-10&1&9\\
7&-8&7&-8&-5&5&3&-8&-8&-6\\
1&9&-10&-4&-8&-5&-1&-2&1&8
\end{pmatrix}.
$$

В таблице 1 мы приводим зависимость скорости работы (время и количество необходимых итераций) от желаемой точности $\varepsilon$.

\begin{table}[]
\centering
\caption{Скорость работы универсального метода.}
\begin{tabular}{|c|c|c|c|c|c|c|c|c|}
\hline
&$\frac{1}{2}$&$\frac{1}{4}$&$\frac{1}{6}$&$\frac{1}{8}$&$\frac{1}{10}$ &$\frac{1}{12}$&$\frac{1}{14}$&$\frac{1}{16}$ \\ \hline
$K$&1157&2082&3268&4140&5528&6426&7396&8458 \\ \hline
Время&0:02:16&0:04:04&0:06:10&0:08:12&0:10:36&0:11:43&0:13:52&0:15:55 \\ \hline
\end{tabular}
\end{table}


Как видим, метод работает со скоростью $O\left(\frac{1}{\varepsilon}\right)$, которая лучше оптимальной теоретической оценки $O\left(\frac{1}{\varepsilon^2}\right)$ для задач уровня гладкости $\nu = 0$.
\end{remark}

\section{О концепции ($\delta$, L)-модели функции для седловых задач}

Вариационные неравенства возникают, в частности, при решении седловых задач, в которых для выпуклого по $u$ и вогнутого по $v$ функционала $f(u,v):\mathbb{R}^{n_1+n_2}\rightarrow\mathbb{R}$ ($u\in Q_1\subset\mathbb{R}^{n_1}$ и $v\in Q_2\subset\mathbb{R}^{n_2}$) требуется найти $(u_*,v_*)$ такую, что:
\begin{equation}\label{eq31}
f(u_*,v)\leqslant f(u_*,v_*)\leqslant f(u,v_*)
\end{equation}
для произвольных $u\in Q_1$ и $v\in Q_2$. Мы считаем $Q_1$ и $Q_2$ выпуклыми компактами в пространствах $\mathbb{R}^{n_1}$ и $\mathbb{R}^{n_2}$ и поэтому $Q=Q_1\times Q_2\subset\mathbb{R}^{n_1+n_2}$ также есть выпуклый компакт. Для всякого $x=(u,v)\in Q$ будем полагать, что
$$||x||=\sqrt{||u||_1^2+||v||_2^2},$$
где $||\cdot||_1$ и $||\cdot||_2$~--- нормы в пространствах $\mathbb{R}^{n_1}$ и $\mathbb{R}^{n_2}$). Условимся обозначать $x=(u_x,v_x),\;y=(u_y,v_y)\in Q$.

Хорошо известно, что для достаточно гладкой функции $f$ по $u$ и $v$ задача \eqref{eq31} сводится к вариационному неравенству с оператором
\begin{equation}\label{eq32}
g(x)=
\begin{pmatrix}
f_u'(u_x,v_x)\\
-f_v'(u_x,v_x)
\end{pmatrix}.
\end{equation}

Предложим некоторую адаптацию концепции $(\delta, L)$-модели, применимую для седловых задач.

\begin{definition}\label{Def_Sedlo}
Будем говорить, что функция $\psi(x,y)\;(\psi:\mathbb{R}^{n_1+n_2}\times\mathbb{R}^{n_1\times n_2}\rightarrow\mathbb{R})$ есть {\bf $(\delta,L)$-модель} для седловой задачи \eqref{eq31}, если для функционала $\psi$ при всяких $x, y, z \in Q$ выполнены предположения:
\begin{enumerate}
\item[(i)] функционал $\psi(x,y)$ --- выпуклый по первой переменной;
\item[(ii)] $\psi(x,x)=0\;\;\forall x\in Q$;
\item[(iii)] ({\it абстрактная монотонность}) неравенство \eqref{eq14};
\item[(iv)] ({\it обобщенная гладкость})
\begin{equation}\label{eqq20}
\psi(x,y)\leqslant\psi(x,z)+\psi(z,y)+LV(x,z)+LV(z,y)+\delta
\end{equation}
для некоторой фиксированной постоянной $L>0$, где $\delta>0$~--- некоторая постоянная величина (оценка погрешности задания $\psi$,
степень отклонения от гладкости);
\item[(v)] справедливо неравенство:
\begin{equation}\label{eq33}
f(u_y,v_x)-f(u_x,v_y)\leqslant-\psi(x,y) \quad \forall x,y\in Q.
\end{equation}
\end{enumerate}
\end{definition}

\begin{example}\label{Example_composite}
Предложенная концепция модели функции для седловых задач вполне применима, например, для рассмотренных в статье \cite{Chambolle} композитных седловых задач вида:
\begin{equation}\label{eq34}
f(u,v)=\tilde{f}(u,v)+h(u)-\varphi(v)
\end{equation}
для некоторой выпуклой по $u$ и вогнутой по $v$ субдифференцируемой функции $\tilde{f}$, а также выпуклых функций простой структуры $h$ и $\varphi$ (для этих функций операция проектирования на множество не очень затратна). В таком случае можно положить
\begin{equation}\label{eq35}
\psi(x,y)=\langle\tilde{g}(y),x-y\rangle+h(u_x)+\varphi(v_x)-h(u_y)-\varphi(v_y),
\end{equation}
где
$$
\tilde{g}(y)=
\begin{pmatrix}
\tilde{f}_u'(u_y,v_y)\\
-\tilde{f}_v'(u_y,v_y)
\end{pmatrix}.
$$
\end{example}

Действительно, из субградиентных неравенств получаем:
$$\tilde{f}(u_y,v_y)-\tilde{f}(u_x,v_y)\leqslant\langle-\tilde{f}_u'(u_y,v_y),u_x-u_y\rangle,$$
$$\tilde{f}(u_y,v_x)-\tilde{f}(u_y,v_y)\leqslant\langle-\tilde{f}_v'(u_y,v_y),v_x-v_y\rangle.$$
Поэтому имеем
$$\tilde{f}(u_y,v_x)-\tilde{f}(u_x,v_y)\leqslant-\langle\tilde{g}(y),x-y\rangle,$$
откуда
$$f(u_y,v_x)-f(u_x,v_y)=\tilde{f}(u_y,v_x)+h(u_y)-\varphi(v_x)-\tilde{f}(u_x,v_y)-h(v_x)+\varphi(v_y)=$$
$$=\tilde{f}(u_y,v_x)-\tilde{f}(u_x,v_y)+h(u_y)+\varphi(v_y)-h(v_x)-\varphi(v_x)\leqslant$$
$$\leqslant-\langle\tilde{g}(y),x-y\rangle+h(u_y)+\varphi(v_y)-h(v_x)-\varphi(v_x)=-\psi(x,y).$$
Из теоремы \ref{th1} вытекает
\begin{theorem}\label{th_sedlo}
Если для седловой задачи \eqref{eq31} существует $(\delta,L)$-модель $\psi(x,y)$, то после остановки алгоритма \ref{alg:SIGM} получаем точку
\begin{equation}\label{eq36}
\tilde{y}=(u_{\tilde{y}},v_{\tilde{y}}):=(\tilde{u},\tilde{v}):=\frac{1}{S_N}\sum_{k=0}^{N_1}\frac{y_{k+1}}{L^{k+1}},
\end{equation}
для которой верна оценка величины-качества решения седловой задачи:
\begin{equation}\label{eq37}
\max_{v\in Q_2}f(\tilde{u},v)-\min_{u\in Q_1}f(u,\tilde{v})\leqslant \varepsilon +2\tilde{\delta}+\delta.
\end{equation}
\end{theorem}

\section{Заключение}

Таким образом, введённая в работе концепция ($\delta, L$)-модели функции для задач равновесного программирования, позволила распространить ранее предложенный в \cite{Ston_2017} (см. также замечание 5.1 из \cite{Gasn_2017}) универсальный метод для вариационных неравенств на более широкий класс задач. В частности, методика настоящей работы применима к смешанным вариационным неравенствам \cite{Konnov_2017, Bao_Khanh}, а также композитным седловым задачам \cite{Chambolle}. При этом была учтена возможность неточного задания функционала $\psi$, а также неточность решения вспомогательных задач проектирования на итерациях метода. Показано, что при этом сохраняются оптимальные с точки зрения нижних оракульных оценок скорости сходимости метода, а погрешности обоих типов не накапливаются в ходе итераций метода. Приводится пример, иллюстрирующий возможность ускорения работы предложенного метода по сравнению с теоретическими оценками за счёт адаптивного выбора шага и адаптивного критерия остановки (см. замечание \ref{Rem_Steiner} выше).

Отметим некоторые открытые вопросы. Хорошо известно, что в общем случае к седловым задачам с помощью метода множителей Лагранжа сводятся задачи выпуклой условной оптимизации. Допустим, что целевой функционал или функционал ограничения в задаче условной оптимизации не имеет липшицева градиента, но при этом для него существует ($\delta, L$)-модель в смысле определения \ref{gen_delta_L_oracle}. В таком случае ожидается, что соответствующая лагранжиева седловая задача будет иметь ($\delta, L$)-модель в смысле определения \ref{Def_Sedlo}. Отметим, что примеры таких функционалов для безусловных задач приведены в \cite{Pap_Model,Ston_Model}. Представляет интерес исследовать применимость результатов настоящей работы к задачам условной оптимизации с такими функционалами. Также представляется актуальным вопрос о возможности использования похожей методики для каких-либо классов задач невыпуклой оптимизации, а также для вариационных неравенств с немонотонными операторами.

Статья опубликована в \cite{StonProc}. Данная версия текста содержит исправления допущенных ранее опечаток.

\end{document}